\documentclass[12pt]{amsart}
\usepackage{amsmath,amssymb,amsthm,color,graphicx,cite,hyperref}
\theoremstyle{plain}

\newtheorem{theorem}{Theorem}[section]

\newtheorem{proposition}[theorem]{Proposition}
\newtheorem{lemma}[theorem]{Lemma}
\newtheorem{observation}[theorem]{Observation}
\newtheorem{corollary}[theorem]{Corollary}

\theoremstyle{definition}

\theoremstyle{remark}

\renewcommand{\epsilon}{\varepsilon}

\newcommand{\F}{\mathcal{F}}

\newcommand{\M}{\mathcal{M}}
\newcommand{\R}{\mathbb{R}}
\newcommand{\T}{\mathcal{T}}

\newcommand{\e}{\varepsilon}
\newcommand{\sset}[1]{\left\{#1\right\}}
\newcommand{\ceil}[1]{\lceil#1\rceil}
\newcommand{\floor}[1]{\lfloor#1\rfloor}

\begin{document}

\title{Piercing axis-parallel boxes}
\author{Maria Chudnovsky}\thanks{M. Chudnovsky:
Princeton University, Princeton, NJ 08544. E-mail: mchudnov@math.princeton.edu. Supported by NSF grant DMS-1550991 and US Army Research Office Grant W911NF-16-1-0404.}
\author{Sophie Spirkl}\thanks{S. Spirkl:
Princeton University, Princeton, NJ 08544. E-mail: sspirkl@math.princeton.edu.}
\author{Shira Zerbib}\thanks{S. Zerbib: Department of Mathematics,
University of Michigan, Ann Arbor. E-mail: zerbib@umich.edu. Partly supported by the New-England Fund, Technion.}


\begin{abstract}
Let $\F$ be a finite family of axis-parallel boxes in $\R^d$ such that $\F$ contains no $k+1$ pairwise disjoint boxes.
We prove that if $\F$ contains a subfamily $\M$ of $k$ pairwise disjoint boxes with the property that for every $F\in \F$ and $M\in \M$ with $F \cap M \neq \emptyset$, either 
$F$ contains a corner of $M$ or 
$M$ contains $2^{d-1}$ corners of $F$, then $\F$ can be pierced by $O(k)$ points. One consequence of this result is that if $d=2$ and the ratio between any of the side lengths of any box is bounded by a constant, then $\F$ can be pierced by $O(k)$ points.        
We further show that if  
for each two intersecting boxes in $\F$ a corner of one is contained in the other, then $\F$ can be pierced by at most $O(k\log\log(k))$ points, and in the special case where $\F$ contains only cubes this bound improves to $O(k)$.  
\end{abstract}

\maketitle

\section{Introduction}
A \emph{matching} in a hypergraph $H=(V,E)$ on  vertex set $V$ and edge set $E$ is a subset of disjoint edges in $E$, and a \emph{cover} of $H$ is a subset of $V$ that intersects all edges in $E$. The \emph{matching number} $\nu(H)$ of $H$ is the maximal size of a matching in $H$, and the \emph{covering number} $\tau(H)$ of $H$ is the minimal size of a cover. The fractional relaxations of these numbers are denoted as usual by $\nu^*(H)$ and $\tau^*(H)$. By LP duality we have that $\nu^*(H)=\tau^*(H)$.

Let $\F$ be a finite family of axis-parallel boxes in $\R^d$. We identify $\F$ with the hypergraph consisting with vertex set $\R^d$ and 
 edge set $\F$. Thus a matching in $\F$ is a subfamily of pairwise disjoint boxes (also called an {\em independent set} in the literature) and a cover in $\F$ is a set of points in $\R^d$ intersecting every box in $\F$ (also called a {\em hitting set}).  
 
An old result due to Gallai is the following (see e.g. \cite{hajnalsuranyi}): 
\begin{theorem}[Gallai]\label{thm:gallai} 
If $\F$ is a family of intervals in $\R$ (i.e., a family of boxes in $\R$) then $\tau(\F) = \nu(\F)$.
\end{theorem} 

A \emph{rectangle} is an axis-parallel box in $\R^2$. In 1965, Wegner \cite{wegner} conjectured that in a hypergraph of axis-parallel rectangles in $\R^2$, the ratio $\tau/\nu$ is bounded by 2. G\'{y}arf\'{a}s and Lehel conjectured in \cite{gyarfaslehel} that the same ratio is bounded by a constant. The best known lower bound, $\tau = \lfloor 5 \nu/3 \rfloor$, is attained by a construction due to Fon-Der-Flaass and Kostochka in \cite{fonderflaass}. K{\'a}rolyi \cite{karolyi} proved that in families of axis-parallel boxes in $\R^d$ we have $\tau(\F) \leq \nu(\F) \left(1 + \log\left(\nu(\F)\right)\right)^{d-1}$, where $\log=\log_2$. Here is a short proof of K{\'a}rolyi's bound.

\begin{theorem}[K\'arolyi \cite{karolyi}] \label{karolyi} If $\F$ is a finite family of axis-parallel boxes in $\R^d$, then $\tau(\F) \leq \nu(\F) \left(1 + \log\left(\nu(\F)\right)\right)^{d-1}$. 
\end{theorem}
\begin{proof}
We proceed by induction on $d$ and $\nu(\F)$. Note that if $\nu(\F) \in \{0, 1\}$ then the result holds for all $d$. Now let $d, n \in \mathbb{N}$. Let $F_{d'}:\R \to \R$ be a function for which $\tau(\T) \le F_{d'}(\nu(\T))$ for every family $\T$ of axis-parallel boxes in $\R^{d'}$ with $d' < d$, or with $d=d'$ and $\nu(\T) < n$.

Let $\F$ be a family of axis-parallel boxes in $\R^d$ with $\nu(\F) = n$. For $a \in \R$, let $H_a$ be the hyperplane $\sset{x = (x_1, \dots, x_d) : x_1=a}$. Write $L_a = \sset{x = (x_1, \dots, x_d) : x_1 \leq a}$, and let $\F_a = \sset{F \in \F: F \subseteq L_a}$. Define $a^* = \min \sset{a : \nu(F_a) \geq \ceil{\nu/2}}$. The hyperplane $H_{a^*}$ gives rise to a partition $\F=\bigcup_{i=1}^3\F_i$, where $\F_1 = \sset{F \in \F : F \subseteq L_{a^*} \setminus H_{a^*}}$, $\F_2 = \sset{F \in \F: F \cap H_{a^*} \neq \emptyset}$, and $\F_3 = \F \setminus (\F_1 \cup F_2)$. It follows from the choice of $a^*$ that $\nu(\F_1) \leq \ceil{\nu(\F)/2}-1$, $\nu(\F_2) \leq \nu(\F)$, and $\nu(\F_3) \leq \floor{\nu(\F)/2}$. 

Therefore, 
\begin{align*}
F_d(\nu(\F)) &\le \tau(\F_1) + \tau(\F_3) + \tau(\sset{F\cap H_{a^*} : F\in \F_2})\\
&\le F_d(\nu(\F_1)) + F_d(\nu(\F_3)) + F_{d-1}(\nu(\F_2))\\
&\le F_d(\Big\lceil\frac{\nu(\F)}{2}\Big\rceil-1) + F_d(\Big\lfloor\frac{\nu(\F)}{2}\Big\rfloor) + F_{d-1}(\nu(\F)) \\
&\le 2~\frac{\nu(\F)}{2}  \left(1 + \log\left(\frac{\nu(\F)}{2}\right)\right)^{d-1} + \nu(\F)  \left(1 + \log\left(\nu(\F)\right)\right)^{d-2} \\
&\le \nu(\F) \left(1 + \log\left(\nu(\F)\right)\right)^{d-1}, 
\end{align*}
implying the result. 
\end{proof} 
For $\nu(\F) = 1$, this implies the following well-known result (see e.\ g.\ \cite{fonderflaass}). 
\begin{observation} \label{obs:pairwise} Let $\F$ be a family of axis-parallel boxes with $\nu(\F) = 1$. Then $\tau(\F) = 1$. 
\end{observation}

Note that for $\nu(\F) = 2$, we have that $\F_1 = \emptyset$, $\nu(\F_2) = 1$ and so $\tau(\F) \leq F_{d-1}(2) + 1$. Therefore, we have the following, which was also proved in~\cite{fonderflaass}. 
\begin{observation}[Fon-der-Flaass and Kostochka \cite{fonderflaass}] \label{obs:nutwo}
Let $\F$ be a family of axis-parallel boxes in $\R^d$ with $\nu(\F) = 2$. Then $\tau(\F) \leq d+1$. 
\end{observation}

The bound from Theorem~\ref{karolyi} was improved by Akopyan~\cite{akopyan} to $\tau(\F) \leq (1.5 \log_3 2 + o(1)) \nu(\F) \left(\log_2\left(\nu(\F)\right)\right)^{d-1}$.

A \emph{corner} of a box $F$ in $\R^d$ is a zero-dimensional face of $F$. 
We say that two boxes in $\R^d$ {\em intersect at a corner} if one of them contains a corner of the other.

A family $\F$ of connected subsets of $\R^2$ is a family of \emph{pseudo-disks}, if for every pair of distinct subsets in $\F$, their boundaries intersect in at most two points. In \cite{chanhar}, Chan and Har-Peled proved that families of pseudo-disks in $\R^2$ satisfy $\tau =O(\nu)$.   It is easy to check that if $\F$ is a family of axis-parallel rectangles in $\R^2$ in which every two intersecting rectangles intersect at a corner, then $\F$ is a family of pseudo-disks. Thus we have:

\begin{theorem}[Chan and Har-Peled \cite{chanhar}]\label{chenhar}
There exists a constant $c$ such that for every family $\F$ of axis-parallel rectangles in $\R^2$ in which every two intersecting rectangles intersect at a corner, we have that $\tau(\F)\le c\nu(\F)$. 
\end{theorem}

Here we prove a few different generalizations of this theorem. In Theorem~\ref{dloglog} we prove the bound  $\tau(\F) \leq  c\nu(\F)\log\log(\nu(\F))$ for families $\F$ of axis-parallel boxes in $\R^d$ in which every two intersecting boxes intersect at a corner, and in Theorem \ref{dloglogsquares} we prove $\tau(\F) \leq c \nu(\F)$ for families $\F$ of  axis-parallel cubes in $\R^d$, where in both cases $c$ is a constant depending only on the dimension $d$. We further prove in Theorem~\ref{mainthm1} that in families $\F$ of axis-parallel boxes in $\R^d$ satisfying certain assumptions on their pairwise intersections, the bound on the covering number improves to $\tau(\F) \leq c \nu(\F)$. For $d=2$, these assumptions are equivalent to the assumption that there is a maximum matching $\M$ in $\F$ such that every intersection between a box in $\M$ and a box in $\F \setminus \M$ occurs at a corner. We use this result to prove our Theorem~\ref{mainthm2}, asserting that for every $r$, if $\F$ is a family of axis-parallel rectangles in $\R^2$ with the property that the ratio between the side lengths of every rectangle in $\F$ is bounded by $r$, then $\tau(\F) \leq c \nu(\F)$ for some constant $c$ depending only on $r$. 

Let us now describe our results in more detail. First, for general dimension $d$ we have the following. 

\begin{theorem}\label{dloglog}
There exists a constant $c$ depending only on $d$, such that for every family $\F$ of axis-parallel boxes in $\R^d$ in which every two intersecting boxes intersect at a corner we have $\tau(\F) \leq  c\nu(\F)\log\log(\nu(\F))$. 
\end{theorem} 

For the proof, we first prove the bound  $\tau^*(\F) \leq  2^d\nu(\F)$ on the fractional covering number of $\F$, and then use Theorem~\ref{nets} below for the bound $\tau(\F) = O(\tau^*(\F)\log\log(\tau^*(\F)))$. 

An axis-parallel box is a \emph{cube} if all its side lengths are equal. Note that if $\F$ consists of axis-parallel cubes in $\R^d$, then every intersection in $\F$ occurs at a corner. Moreover, for axis-parallel cubes we have $\tau(\F) = O(\tau^*(\F))$ by Theorem~\ref{nets}, and thus we conclude the following.

\begin{theorem}\label{dloglogsquares}
If $\F$ be a family of axis-parallel cubes in $\R^d$, then $\tau(\F) \leq  c\nu(\F)$ for some constant $c$ depending only on $d$. 
\end{theorem}

To get a constant bound on the ratio $\tau/\nu$ in families of axis-parallel boxes in $\R^d$ which are not necessarily cubes, we make a more restrictive assumption on the intersections in $\F$. 

\begin{theorem}\label{mainthm1}
Let $\F$ be a family of axis-parallel boxes in $\R^d$. Suppose that there exists a maximum matching $\M$ in $\F$ such that for every $F\in \F$ and $M\in \M$, at least one of the following holds: 
\begin{enumerate}
\item $F$ contains a corner of $M$; 
\item $F \cap M = \emptyset$;  or
\item $M$ contains $2^{d-1}$ corners of $F$. 
\end{enumerate}
Then $\tau(\F) \leq (2^d+(4+d)d) \nu(\F)$. 
\end{theorem}

For $d=2$, this theorem implies the following corollary.  
\begin{corollary}\label{cor1}
Let $\F$ be a family of axis-parallel rectangles in $\R^2$. Suppose that there exists a maximum matching $\M$ in $\F$ such that for every $F\in \F$ and $M\in \M$, if $F$ and $M$ intersect then they intersect at a corner. 
Then $\tau(\F) \leq 16\nu(\F)$.
\end{corollary}

Note that Corollary \ref{cor1} is slightly stronger than Theorem \ref{chenhar}. Here we only need that the intersections with rectangles in some fixed maximum matching $\M$ occur at corners, but we do not restrict the intersections of two rectangles $F,F' \notin \M$.  

Given a constant $r>0$, we say that a family $\F$ of axis-parallel boxes in $\R^d$ has an \emph{$r$-bounded aspect ratio} if every box $F\in \F$ has $l_i(F)/l_j(F) \le r$ for all $i, j \in \sset{1,\dots, d}$, where $l_i(F)$ is the length of the orthogonal projection of $F$ onto the $i$th coordinate. 
 
For families of rectangles with bounded aspect ratio we prove the following. 

\begin{theorem}\label{mainthm2}
Let $\F$ be a family of axis-parallel rectangles in $\R^2$ that has an $r$-bounded aspect ratio. Then $\tau(\F) \leq (14+2r^2)\nu(\F)$.
\end{theorem}

A result similar to Theorem \ref{mainthm2} was announced in \cite{ahlkar}, but to the best of our knowledge the proof was not published. 

An application of Theorem \ref{mainthm2} is 
the existence of weak $\e$-nets of size $O \left(\frac{1}{\e} \right)$ for axis-parallel rectangles in $\R^2$ with bounded aspect ratio. More precisely, let $P$ be a set of $n$ points in $\R^d$ and let $\F$ be a family of sets in $\R^d$, each containing at least $\e n$ points of $P$.  
A \emph{weak $\e$-net} for $\F$ is a cover of $\F$, and a \emph{strong $\e$-net} for $\F$ is a cover of $\F$ with points of $P$. The existence of weak $\e$-nets of size $O \left(\frac{1}{\e} \right)$ for pseudo-disks in $\R^2$ was proved by Pyrga and Ray in \cite{pygra}. Aronov, Ezra and Sharir in \cite{aronovezrasharir} showed the existence of strong $\e$-nets of size $O \left( \frac{1}{\e} \log \log \frac{1}{\e} \right)$ for axis-parallel boxes in $\R^2$ and $\R^3$, and the existence of weak $\e$-nets of size $O \left( \frac{1}{\e} \log \log \frac{1}{\e} \right)$ for all $d$ was then proved by Ezra in \cite{ezra}. Ezra also showed that for axis-parallel cubes in $\R^d$ there exist an $\e$-net of size $O \left(\frac{1}{\e} \right)$. These results imply the following.

\begin{theorem}[Aronov, Ezra and Sharir \cite{aronovezrasharir}; Ezra \cite{ezra}]\label{nets}
If $\F$ is a family of axis-parallel boxes in $\R^d$ then $\tau(\F) \leq  c\tau^*(\F)\log\log(\tau^*(\F))$ for some constant $c$ depending only on $d$. If $\F$ consists of cubes, then this bound improves to $\tau(\F) \leq  c\tau^*(\F)$. 
\end{theorem}

An example where the smallest strong $\e$-net for axis-parallel rectangles in $R^2$ is of size $\Omega \left(\frac{1}{\e} \log \log \frac{1}{\e} \right)$ was constructed by Pach and Tardos in \cite{pachtardos}. The question of whether weak $\e$-nets of size $O( \frac{1}{\e})$ for axis-parallel rectangles in $\R^2$ exist was raised both in \cite{aronovezrasharir} and in  \cite{pachtardos}. 

Theorem \ref{mainthm2} implies a positive answer for the family of axis-parallel rectangles in $\R^2$ satisfying the $r$-bounded aspect ratio property: 

\begin{corollary}\label{cor:aspect}
For every fixed constant $r$, there exists a weak $\e$-net of size $O(\frac{1}{\e})$ for the family $\F$ of axis-parallel rectangles in $\R^2$ with aspect ratio bounded by $r$. 
\end{corollary}
\begin{proof}
Given a set $P$ of $n$ points, there cannot be $\frac{1}{\e} + 1$ pairwise disjoint rectangles in $\F$, each containing at least $\e n$ points of $P$. Therefore $\nu(\F) \le \frac{1}{\e}$. Theorem~\ref{mainthm2} implies that there is a cover of $\F$ of size $O(\frac{1}{\e})$. 
\end{proof}

This paper is organized as follows. In Section 2 we prove Theorem \ref{dloglog}. Section 3 contains definitions and tools. Theorem \ref{mainthm1} is then proved in Section 4 and Theorem \ref{mainthm2} is proved in Section 5.

\section{Proofs of Theorems~\ref{dloglog} and \ref{dloglogsquares}}

Let $\F$ be a finite family of axis-parallel boxes in $\R^d$, such that every intersection in $\F$ occurs at a corner. By performing small perturbations on the boxes, we may assume that no two corners of boxes of $\F$ coincide.

\begin{proposition}
We have $\tau^*(\F) \le 2^d\nu(\F)$.  
\end{proposition}
\begin{proof}
Write $\nu(\F) = k$, and let $g:\F \to \mathbb{R}^+$ be a rational approximation of a maximal fractional matching for $\F$. By removing boxes $F\in \F$ for which $g(F)=0$
and duplicating boxes if necessary,
we may assume that $g(F)=\frac{1}{r}$ for all $F
\in \F$, where $r$ is the maximal size of a subset of boxes in $\F$ intersecting in a single point. Letting $n$ be the number of boxes in $\F$ we have $\tau^*(\F)=\nu^*(\F)=\frac{n}{r}$, and thus our aim is to show that
$\frac{n}{r} \le 2^dk$.

Since $\nu(\F)=k$, it follows from Tur\'{a}n's theorem that there are at least $n(n - k)/(2k)$ unordered intersecting pairs of boxes $\F$. Each such unordered pair contributes at least two pairs of the form $(x,F)$, where $x$ is
a corner of a box $F'\in \F$, $F$ is box in $\F$ different from $F'$, and $x$ pierces $F$. Therefore, since there are altogether $2^dn$ corners of boxes in $\F$, there must exist a corner $x$   
of a box $F\in \F$ that pierces at least $(n - k)/2^dk$ boxes in $\F$, all different from $F$. Together with $F$, $x$ pierces at least $n/2^dk$ boxes, implying that $n/2^dk \le r$. Thus $\frac{n}{r} \le 2^dk$, as desired.  
\end{proof}

Combining this bound with Theorem \ref{nets}, we obtain the proofs of Theorems \ref{dloglog} and \ref{dloglogsquares}.

\section{Definitions and tools}
Let $R$ be an axis-parallel box in $\R^d$ with $R = [x_1, y_1] \times \dots \times [x_d, y_d]$. For $i \in \sset{1, \dots, d}$, let $p_i(R) = [x_i, y_i]$ denote the  orthogonal projection of $R$ onto the $i$-th coordinate. Two intervals $[a,b], [c,d] \subseteq \R$, are \emph{incomparable} if $[a,b] \not\subseteq [c,d]$ and $[c,d] \not\subseteq [a,b]$. We say that $[a, b] \prec [c, d]$ if $b < c$. For two axis-parallel boxes $Q$ and $R$ we say that $Q \prec_i R$ if $p_i(Q) \prec p_i(R)$. 

\begin{observation} \label{obs:disjoint} Let $Q, R$ be disjoint axis-parallel boxes in $\R^d$. Then there exists $i \in \sset{1, \dots, d}$ such that $Q \prec_i R$ or $R \prec_i Q$. 
\end{observation}

\begin{lemma} \label{lem:contains} Let $Q, R$ be axis-parallel boxes in $\R^d$ such that $Q$ contains a corner of $R$ but $R$ does not contain a corner of $Q$. Then, for all $i \in \sset{1, \dots, d}$, either $p_i(R)$ and $p_i(Q)$ are incomparable, or $p_i(R) \subseteq p_i(Q)$, and there exists  $i \in \sset{1, \dots, d}$ such that $p_i(R) \subsetneq p_i(Q)$. 

Moreover, if $R \not\subseteq Q$, then there exists  $j \in \sset{1, \dots, d} \setminus \sset{i}$ such that $p_i(R)$ and $p_i(Q)$ are incomparable. 
\end{lemma}
\begin{proof}
Let $x = (x_1, \dots, x_d)$ be a corner of $R$ contained in $Q$. By symmetry, we may assume that $x_i = \max(p_i(R))$ for all $i \in \sset{1, \dots, d}$. Since $x_i \in p_i(Q)$ for all $i \in \sset{1, \dots, d}$, it follows that $\max(p_i(Q)) \geq \max(p_i(R))$ for all $i \in \sset{1, \dots, d}$. If $\min(p_i(Q)) \leq \min(p_i(R))$, then $p_i(R) \subseteq p_i(Q)$; otherwise, $p_i(Q)$ and $p_i(R)$ are incomparable. If $p_i(Q)$ and $p_i(R)$ are incomparable for all $i \in \sset{1, \dots, d}$, then $y = (y_1, \dots, y_d)$ with $y_i = \min(p_i(Q))$ is a corner of $Q$ and since $\min(p_i(Q)) > \min(p_i(R))$, it follows that $y \in R$, a contradiction. It follows that there exists an $i \in \sset{1, \dots, d}$ such that $p_i(R) \subsetneq p_i(Q)$. 

If $p_i(R) \subsetneq p_i(Q)$ for all $i \in \sset{1, \dots, d}$, then $R \subseteq Q$; this implies the result.
\end{proof}

\begin{observation} \label{lem:contains2} Let $\F$ be a family of axis-parallel boxes in $\R^d$. Let $\F'$ arise from $\F$ by removing every box in $\F$ that contains another box in $\F$. Then $\nu(\F) = \nu(\F')$ and $\tau(\F) = \tau(\F')$. 
\end{observation}
\begin{proof}
Since $\F' \subseteq \F$, it follows that $\nu(\F') \leq \nu(\F)$ and $\tau(\F') \leq \tau(\F)$. Let $\M$ be a matching in $\F$ of size $\nu(\F)$. Let $\M'$ arise from $\M$ by replacing each box $R$ in $\M \setminus \F'$ with a box in $\F'$ contained in $R$. Then $\M'$ is a matching in $\F'$, and so $\nu(\F') = \nu(\F)$. Moreover, let $P$  be a cover of $\F'$. Since every box in $\F$ contains a box in $F'$ (possibly itself) which, in turn, contains a point in $P$, we deduce that $P$ is a cover of $\F$. It follows that $\tau(\F') = \tau(\F)$. 
\end{proof}

A family $\F$ of axis-parallel boxes is \emph{clean} if no box in $\F$ contains another box in $\F$. By Observation~\ref{lem:contains2}, we may restrict ourselves to clean families of boxes. 

\section{Proof of Theorem \ref{mainthm1}}

Throughout this section, let $\F$ be a clean family of axis-parallel boxes in $\R^d$, and let $\M$ be a matching of maximum size in $\F$. We let $\F(\M)$ denote the subfamily of $\F$ consisting of those boxes $R$ in $\F$ for which for every $M \in \M$, either $M$ is disjoint from $R$ or $M$ contains at least $2^{d-1}$ corners of $R$. Our goal is to bound $\tau(\F(\M))$.

\begin{lemma} \label{lem:intersect} Let $R \in \F(\M)$. Then $R$ intersects at least one and at most two boxes in $\M$. If $R$ intersects two boxes $M_1, M_2 \in \M$, then there exists $j \in \sset{1,\dots, d}$ such that $M_1 \prec_j M_2$ or $M_2 \prec_j M_1$, and for all $i \in \sset{1, \dots, d} \setminus \sset{j}$, we have that $p_i(R) \subseteq p_i(M_1)$ and $p_i(R) \subseteq p_i(M_2)$.
\end{lemma}
\begin{proof}
If $R$ is disjoint from every box in $\M$, then $\M \cup \sset{R}$ is a larger matching, a contradiction. So $R$ intersects at least one box in $\M$. Let $M_1$ be in $\M$ such that $R \cap M_1 \neq \emptyset$. We claim that there exists $j \in \sset{1, \dots, d}$ such that $M_1$ contains precisely the set of corners of $R$ with the same $j$th coordinate. 

By Lemma~\ref{lem:contains}, there exists $j \in \sset{1,\dots,d}$ such that $p_j(R) = [a,b]$ and $p_j(M_1)$ are incomparable. By symmetry, we may assume that $a \in p_j(M_1)$, $b \not\in p_j(M_1)$. This proves that $M_1$ contains all $2^{d-1}$ corners of $R$ with $a$ as their $j$th coordinate, and our claim follows.

Consequently, $p_i(R) \subseteq p_i(M_1)$ for all $i \in \sset{1, \dots, d} \setminus \sset{j}$. Since $R$ has exactly $2^d$ corners, and members of $\M$ are disjoint, it follows that there exist at most two boxes in $\M$ that intersect $R$. If $M_1$ is the only one such box, then the result follows. Let $M_2 \in \M \setminus \sset{M_1}$ such that $R \cap M_1 \neq \emptyset$. By our claim, it follows that $M_2$ contains $2^{d-1}$ corners of $R$; and since $M_1$ is disjoint from $M_2$, it follows that $M_2$ contains precisely those corners of $R$ with $j$th coordinate equal to $b$. Therefore, $p_i(R) \subseteq p_i(M_2)$ for all $i \in \sset{1, \dots, d} \setminus \sset{j}$.  We conclude that $p_i(M_2)$ is not disjoint from $p_i(M_1)$ for all $i \in \sset{1, \dots, d} \setminus \sset{j}$, and  since $M_1, M_2$ are disjoint, it follows from Observation~\ref{obs:disjoint} that either $M_1 \prec_j M_2$ or $M_2 \prec_j M_1$. 
\end{proof}

For $i \in \sset{1, \dots, d}$, we define a directed graph $G_i$ as follows. We let $V(G_i) = \M$, and for $M_1, M_2 \in \M$ we let $M_1M_2 \in E(G_i)$ if and only if $M_1 \prec_i M_2$ and there exists $R \in \F(\M)$ such that $R \cap M_1 \neq \emptyset$ and $R \cap M_2 \neq \emptyset$. In this case, we say that $R$ \emph{witnesses} the edge $M_1M_2$. For $i = \sset{1, \dots, d}$, we say that $R$ is \emph{$i$-pendant} at $M_1 \in \M$ if $M_1$ is the only box of $\M$ intersecting $R$ and $p_i(R)$ and $p_i(M_1)$ are incomparable. Note that by Lemma~\ref{lem:intersect}, every box $R$ in $\F(\M)$ satisfies exactly one of the following:  
$R$ witnesses an edge in exactly one of the graphs $G_i$, $i \in \sset{1,\dots,d}$; or
 $R$ is $i$-pendant for exactly one $i \in \sset{1,\dots, d}$.

\begin{lemma} \label{lem:edges} Let $i \in \sset{1,\dots,d}$. Let $Q, R \in \F(\M)$ be such that $Q$ witnesses an edge $M_1M_2$ in $G_i$, and $R$ witnesses an edge $M_3M_4$ in $G_i$. If $Q$ and $R$ intersect, then either $M_1 = M_4$, or $M_2 = M_3$, or $M_1M_2 = M_3M_4$. 
\end{lemma}
\begin{proof} By symmetry, we may assume that $i=1$. Let $p_1(M_1) = [x_1,y_1]$ and $p_1(M_2) = [x_2,y_2]$. It follows that $p_1(Q) \subseteq [x_1,y_2]$. Let $a = (a_1, a_2, \dots, a_d) \in Q \cap R$. It follows that $a_j \in p_j(Q) \subseteq p_j(M_1) \cap p_j(M_2)$ and $a_j \in p_j(R) \subseteq p_j(M_3) \cap p_j(M_4)$ for all $j \in \sset{2, \dots, d}$. 

If $M_1 \in \sset{M_3, M_4}$ and $M_2 \in \sset{M_3, M_4}$, then $M_1M_2 = M_3M_4$, and the result follows. Therefore, we may assume that this does not happen. If $M_1 \in \sset{M_3, M_4}$, we reflect every rectangle in $\F$ along the origin. When constructing $G_1$ for this family, we have $M_4M_3, M_2M_1 \in E(G_1)$, and $M_2 \not\in \sset{M_3, M_4}$. Thus, by symmetry, we may assume that $M_1$ is distinct from $M_3$ and $M_4$. 

It follows that $a \not\in M_1$, for otherwise $R$ intersects three distinct members of $\M$, contrary to Lemma~\ref{lem:intersect}. Since $R$ is disjoint from $M_1$, it follows that either $M_1 \prec_1 R$ or $R \prec_1 M_1$. But $p_1(Q) \subseteq [x_1,y_2]$, and since $Q \cap R \neq \emptyset$, it follows that $M_1 \prec_1 R$.

Since $M_3 \neq M_1$ and $p_j(M_3) \cap p_j(M_1) \ni a_j$ for all $j \in \sset{2, \dots, d}$, it follows that either $M_1 \prec_1 M_3$ or $M_1 \prec_1 M_3$. Since $M_1 \prec_1 R$ and $R \cap M_3 \neq \emptyset$, it follows that $M_1 \prec_1 M_3$. 

Suppose that $a \in M_3$. Then $Q \cap M_3 \neq \emptyset$, and since $M_1 \prec_1 M_3$, we have that $M_3 = M_2$ as desired. 

Therefore, we may assume that $a \not\in M_3$, and thus $p_1(M_1) \prec p_1(M_3) \prec [a_1,a_1]$. Since $[y_1, a_1] \subseteq p_1(Q)$, it follows that $p_1(M_3) \cap p_1(Q) \neq \emptyset$. But $p_j(M_3) \cap p_j(Q) \ni a_j$ for all $j \in \sset{2, \dots, d}$, and hence $Q \cap M_3 \neq \emptyset$. But then $M_3 \in \sset{M_1, M_2}$, and thus $M_3 = M_2$. This concludes the proof. 
\end{proof}

The following is a well-known fact about directed graphs; we include a proof for completeness. 
\begin{lemma} \label{lem:digraph} Let $G$ be a directed graph. Then there exists an edge set $E \subseteq E(G)$ with $|E| \geq |E(G)|/4$ such that for every vertex $v \in V(G)$, either $E$ contains no incoming edge at $v$, or $E$ contains no outgoing edge at $v$. 
\end{lemma}
\begin{proof}
For $A, B \subseteq V(G)$, let $E(A, B)$ denote the set of edges of $G$ with head in $A$ and tail in $B$. 

Let $X_0 = Y_0 = \emptyset$, $V(G) = \sset{v_1, \dots, v_n}$. For $i=1, \dots, n$ we will construct $X_i, Y_i$ such that $X_i \cup Y_i = \sset{v_1, \dots, v_i}$, $X_i \cap Y_i = \emptyset$ and $|E(X_i, Y_i)| + |E(Y_i, X_i)| \geq |E(G|(X_i \cup Y_i))|/2$, where $G|(X_i \cup Y_i)$ denotes the induced subgraph of $G$ on vertex set $X_i \cup Y_i$. This holds for $X_0, Y_0$. Suppose that we have constructed $X_{i-1}, Y_{i-1}$ for some $i \in \sset{1, \dots, n}$. If $|E(X_{i-1}, \sset{v_i})| + |E(\sset{v_i}, X_{i-1})| \geq |E(Y_{i-1}, \sset{v_i})| + |E(\sset{v_i}, Y_{i-1})|$, we let $X_i = X_{i-1}, Y_i = Y_{i-1} \cup \sset{v_i}$; otherwise, let $X_i = X_{i-1} \cup \sset{v_i}, Y_i = Y_{i-1}$. It follows that $X_i, Y_i$ still have the desired properties. Thus, $|E(X_n, Y_n)| + |E(Y_n, X_n)| \geq |E(G)|/2$. By symmetry, we may assume that $|E(X_n, Y_n)| \geq |E(G)|/4$. But then $E(X_n, Y_n)$ is the desired set $E$; it contains only incoming edges at vertices in $X_n$, and only outgoing edges at vertices in $Y_n$. This concludes the proof. 
\end{proof}

\begin{theorem}\label{thm:edges} For $i \in \sset{1,\dots, d}$, $|E(G_i)| \leq 4 \nu(\F)$. 
\end{theorem}
\begin{proof}
Let $E \subseteq E(G_i)$ as in Lemma~\ref{lem:digraph}. For each edge in $E$, we pick one box witnessing this edge; let $\F'$ denote the family of these boxes. We claim that $\F'$ is a matching. Indeed, suppose not, and let $Q, R \in \F'$ be distinct and intersecting. Let $Q$ witness $M_1M_2$ and $R$ witness $M_3M_4$. By Lemma~\ref{lem:edges}, it follows that either $M_1M_2 = M_3M_4$ (impossible since we picked exactly one witness per edge) or $M_1 = M_4$ (impossible because $E$ does not contain both an incoming and an outgoing edge at $M_1 = M_4$) or  $M_2 = M_3$ (impossible because $E$ does not contain both an incoming and an outgoing edge at $M_2 = M_3$). This is a contradiction, and our claim follows. Now we have $\nu(\F) \geq |\F'| = |E| \geq |E(G)|/4$, which implies the result. 
\end{proof}

A matching $\M$ of a clean family $\F$ of boxes is \emph{extremal} if for every $M \in \M$ and $R \in \F \setminus \M$, either $(\M \setminus \sset{M}) \cup \sset{R}$ is not a matching or there exists an $i \in \sset{1, \dots, d}$ such that $\max(p_i(R)) \geq \max (p_i(M))$. Every family $\F$ of axis parallel boxes has an extremal maximum matching. For example, the maximum matching $\M$ minimizing $\sum_{M \in \M} \sum_{i=1}^d \max(p_i(M))$ is extremal. 

\begin{theorem} \label{thm:cover} For $i \in \sset{1,\dots, d}$, let $\F_i$ denote the set of boxes in $\F(\M)$ that either are $i$-pendant or witness an edge in $G_i$. Then $\tau(\F_i) \leq (4+d)\nu(\F)$. If $\M$ is extremal, then  $\tau(\F_i) \leq (3+d)\nu(\F)$.
\end{theorem}
\begin{proof}
By symmetry, it is enough to prove the theorem for $i=1$. For $M \in \M$, let $\F_M$ denote the set of boxes in $\F_1$ that either are 1-pendant at $M$, or witness an edge $MM'$ of $G_1$. It follows that $\bigcup_{M \in \M} \F_M = \F_1$. For $M \in \M$, let $d^+(M)$ denote the out-degree of $M$ in $G_1$. We will prove that $\tau(\F_M) \leq d^+(M) + d$ for all $M \in \M$. 

Let $M \in \M$, and let $\mathcal{A}$ denote the set of boxes that are $1$-pendant at $M$. Suppose that $\mathcal{A}$ contains two disjoint boxes $M_1, M_2$. Then $(\M \setminus \sset{M}) \cup \sset{M_1, M_2}$ is a larger matching than $\M$, a contradiction. So every two boxes in $\mathcal{A}$ pairwise intersect. By Observation~\ref{obs:pairwise}, it follows that $\tau(A) = 1$. 

Let $\mathcal{B} = \F_M \setminus \mathcal{A}$. Suppose that there is an edge $MM' \in E(G_1)$ such that the set $\mathcal{B}(M')$ of boxes in $\mathcal{B}$ that witness the edge $MM'$ satisfies $\nu(\mathcal{B}(M')) \geq 3$. Then $\M$ is not a maximum matching, since removing $M$ and $M'$ from $\M$ and adding $\nu(\mathcal{B}(M'))$ disjoint rectangles in $\mathcal{B}(M')$ yields a larger matching.  Moreover, for distinct $M', M'' \in \M$, every box in $\mathcal{B}(M')$ is disjoint from every box in $\mathcal{B}(M'')$ by Lemma~\ref{lem:edges}. Thus, if there exist $M', M''$ such that $\nu(\mathcal{B}(M')) = \nu(\mathcal{B}(M'')) = 2$ and $M' \neq M''$, then removing $M, M'$ and $M''$ and adding two disjoint rectangles from each of $\mathcal{B}(M')$ and $\mathcal{B}(M'')$ yields a bigger matching, a contradiction. 

Let $p_1(M) = [a,b]$. Two boxes in $\mathcal{B}(M')$ intersect if and only if their intersections with the hyperplane $H = \sset{(x_1, \dots, x_d) : x_1 = b}$ intersect. If $\nu(\mathcal{B}(M')) = 1$, then $\tau(\mathcal{B}(M')) = 1$ by Observation~\ref{obs:pairwise}. If $\nu(\mathcal{B}(M')) = 2$, then $\nu(\sset{F \cap H : F \in \mathcal{B}(M')}) = 2$ and so $$\tau(\mathcal{B}(M')) = \tau(\sset{F \cap H : F \in \mathcal{B}(M')}) \leq d$$ by Observation~\ref{obs:nutwo}. 

Therefore, $$\tau(\mathcal{B}) \leq \sum_{M': MM' \in E(G_1)} \tau(\mathcal{B}(M')) \leq d^+(M) - 1 + d,$$ and since $\tau(\mathcal{A}) \leq 1$, it follows that $\tau(\F_M) \leq d^+(M) + d$ as claimed.  

Summing over all rectangles in $\M$, we obtain
\begin{align*}
\tau(\F_i) &\leq \sum_{M \in \M} \tau(\F_M) \leq \sum_{M \in \M} (d^+(M) + d) \\
&= d|V(G_1)| + |E(G_1)|\leq d|\M| + 4|\M| = (4+d)\nu(\F),
\end{align*}
which proves the first part of the theorem.

If $\M$ is extremal, then every $1$-pendant box at $M$ also intersects $H$. Let $M'$ be such that $\nu(\mathcal{B}(M'))$ is maximum. It follows that $\nu(\mathcal{A} \cup \mathcal{B}(M')) \leq 2$ and thus $\tau(\mathcal{A} \cup \mathcal{B}(M')) \leq d$, implying  $\tau(\F_M) \leq d^+(M) + d -1$. 
This concludes the proof of the second part of the theorem.
\end{proof}

\begin{theorem} \label{maintechnical}  Let $\F' \subseteq \F$ be the set of boxes $R \in \F$ such that for each $M \in \M$, either $M \cap R = \emptyset$, or $M$ contains $2^{d-1}$ corners of $R$, or $R$ contains a corner of $M$. Then $\tau(\F') \leq  (2^d+(4+d)d)   \nu(\F)$. If $\M$ is extremal, then  $\tau(\F') \leq  (2^d+(3+d)d)   \nu(\F)$.
\end{theorem}
\begin{proof}
We proved in Theorem~\ref{thm:cover} that $\tau(\F_i) \leq (4+d)\nu(\F)$ for $i=1,\dots,d$. Let $\mathcal{F}'' = \F' \setminus \F(\M)$. Then $\F''$ consists of boxes $R$ such that $R$ contains a corner of some box $M \in \M$. Let $P$ be the set of all corners of boxes in $\M$. It follows that $P$ covers $\F''$, and so $\tau(\F'') \leq 2^d\nu(\F)$. Since $\F' = \F'' \cup \F_1 \cup \dots \cup \F_d$, it follows that $\tau(\F') \leq (2^d+(4+d)d)  \nu(\F)$. If $\M$ is extremal, the same argument yields that $\tau(\F') \leq (2^d+(3+d)d)  \nu(\F)$, since $\tau(\F_i) \leq (3+d)\nu(\F)$ for $i = 1\dots, d$ by Theorem~\ref{thm:cover}. 
\end{proof}

We are now ready to prove our main theorems. 
\begin{proof}[Proof of Theorem~\ref{mainthm1}] Let $\F$ be a family of axis-parallel boxes in $\R^d$, and let $\M$ be a maximum matching in $\F$ such that for every $F \in \F$ and $M \in \M$, either $F \cap M = \emptyset$, or $F$ contains a corner of $M$, or $M$ contains $2^{d-1}$ corners of $F$. It follows that $\F = \F'$ in Theorem~\ref{maintechnical}, and therefore, $\tau(\F) \leq (2^d+(4+d)d) \nu(\F)$. 
\end{proof}

\section{Proof of Theorem \ref{mainthm2}}

Let $\M$ be a maximum matching in $\F$, and let $\M$ be extremal. Observe that each rectangle $R \in \F$ satisfies one of the following:
\begin{itemize}
\item $R$ contains a corner of some $M \in \M$; 
\item some $M \in \M$ contains two corners of $R$; or
\item there exists $M \in \M$ such that $M \cap R \neq \emptyset$, and $p_i(R) \supseteq p_i(M)$ for some $i \in \sset{1,2}$.
\end{itemize}
By Theorem~\ref{maintechnical}, $14 \nu(\F)$ points suffice to cover every rectangle satisfying at least one of the first two conditions. Now, due to the $r$-bounded aspect ratio, for each $M \in 
\M$ and for each $i\in \{1,2\}$, at most $r^2$ disjoint rectangles $R\in \F$ can satisfy the third condition for $M$ and $i$. Thus the family of projections of the rectangles satisfying the third condition for $M$ and $i$ onto the $(3-i)$th coordinate have a matching number at most $r^2$. Since all these rectangles intersect the boundary of $M$ twice,  by Theorem~\ref{thm:gallai}, we need at most $r^2$ additional points to cover them. We conclude that $\tau(\F) \leq (14+2r^2) \nu(\F)$. 
\qed

\section*{Acknowledgments}
We are thankful to Paul Seymour for many helpful discussions. 
\bibliography{piercingd}{}
\bibliographystyle{plain}
\end{document}